\documentclass[a4,12pt]{amsart}
\usepackage{amssymb,amstext,amsmath,amscd,amsthm,amsfonts,enumerate,graphicx,latexsym,color,stmaryrd}
\usepackage[all]{xy}
\usepackage{comment}
\newcommand{\comm}[1]{{\color{green} #1}}
\newcommand{\old}[1]{{\color{red} #1}}

\newtheorem{thm}{Theorem}[section]
\newtheorem*{thm*}{Theorem}

\newtheorem{cor}[thm]{Corollary}
\newtheorem{lem}[thm]{Lemma}
\newtheorem{prop}[thm]{Proposition}
\theoremstyle{definition}

\newtheorem{dfn}[thm]{Definition}
\newtheorem*{dfn*}{Definition}
\newtheorem{rem}[thm]{Remark}
\newtheorem{ques}[thm]{Question}

\newtheorem*{conj*}{Conjecture}
\newtheorem{ex}[thm]{Example}
\newtheorem{nota}[thm]{Notation}

\theoremstyle{remark}
\newtheorem*{ac}{Acknowledgments}
\newtheorem{claim}{Claim}
\newtheorem*{claim*}{Claim}
\renewcommand{\qedsymbol}{$\blacksquare$}
\numberwithin{equation}{thm}
\def\cm{\mathsf{CM}}
\def\lcm{\mathsf{\underline{CM}}}
\def\lmod{\operatorname{\mathsf{\underline{mod}}}}
\def\dub{\mathsf{D}}
\def\db{\mathsf{D^b}}
\def\dpf{\mathsf{D^{perf}}}

\def\ds{\mathsf{D_{sg}}}
\def\kb{\mathsf{K^b}}

\def\proj{\mathsf{proj}}

\def\mod{\operatorname{\mathsf{mod}}}

\def\cx{\mathsf{cx}}
\def\pd{\mathsf{pd}}
\def\id{\mathsf{id}}
\def\codim{\mathsf{codim}}

\def\ltensor{\otimes^{\mathbf{L}}}
\def\rhom{\mathbf{R} \mathcal{H}om}

\def\syz{\Omega}

\def\thick{\operatorname{\mathsf{thick}}}
\def\res{\mathsf{res}}

\def\Hom{\mathsf{Hom}}
\def\End{\mathsf{End}}
\def\lHom{\underline{\mathsf{Hom}}}
\def\lEnd{\underline{\mathsf{End}}}
\def\Ext{\mathsf{Ext}}
\def\Tor{\mathsf{Tor}}
\def\Ker{\operatorname{\mathsf{Ker}}}

\def\cok{\mathsf{Coker}}

\def\nf{\mathsf{NF}}

\def\spec{\operatorname{\mathsf{Spec}}}
\def\spc{\operatorname{\mathsf{Spc}}}
\def\Proj{\operatorname{\mathsf{Proj}}}
\def\supp{\mathsf{Supp}}
\def\spp{\mathsf{Spp}}
\def\lsupp{\underline{\mathsf{Supp}}}
\def\ssupp{\mathsf{SSupp}}
\def\sing{\operatorname{\mathsf{Sing}}}

\def\sig{\sigma}
\def\del{\delta}
\def\d{\partial}

\def\m{\mathfrak{m}}
\def\n{\mathfrak{n}}
\def\p{\mathfrak{p}}
\def\q{\mathfrak{q}}

\def\F{\mathcal{F}}

\def\H{\mathsf{H}}
\def\O{\mathcal{O}}
\def\P{\mathcal{P}}
\def\T{\mathcal{T}}
\def\TT{\mathsf{T}}
\def\U{\mathcal{U}}
\def\V{\mathsf{V}}
\def\X{\mathcal{X}}
\def\Y{\mathcal{Y}}
\def\Z{\mathbb{Z}}

\def\th{\operatorname{\mathsf{Th}}}
\def\pth{\operatorname{\mathsf{PTh}}}
\def\irr{\operatorname{\mathsf{Irr}}}
\def\spcl{\operatorname{\mathsf{Spcl}}}
\def\thom{\operatorname{\mathsf{Thom}}}
\def\cl{\operatorname{\mathsf{Cl}}}
\def\nesc{\operatorname{\mathsf{Nesc}}}

\def\one{\mathbf{1}}

\def\a{\mathsf{A}}
\def\d{\mathsf{D}}

\def\ttthick{\mathsf{thick}^\otimes}
\title[Singular equivalences and reconstruction of singular loci]{singular equivalences of commutative noetherian rings and reconstruction of singular loci}

\author{Hiroki Matsui} 
\address{Graduate School of Mathematics, Nagoya University, Furocho, Chikusaku, Nagoya, Aichi 464-8602, Japan}
\email{m14037f@math.nagoya-u.ac.jp}
\date{}
\thanks{The author is partly supported by Grant-in-Aid for JSPS Fellows 16J01067.
}
\begin{document}

\begin{abstract}
Two left noetherian rings $R$ and $S$ are said to be {\it singularly equivalent} if their singularity categories are equivalent as triangulated categories.
The aim of this paper is to give a necessary condition for two commutative noetherian rings to be singularly equivalent.
To do this, we develop the support theory for triangulated categories without tensor structure.
\end{abstract}

\maketitle

\section{Introduction}
Let $R$ be a left noetherian ring.
The {\it singularity category} of $R$ is by definition the Verdier quotient
$$
\mathsf{D_{sg}}(R):=\mathsf{D^b}(\mathsf{mod} R)/\kb(\proj R),
$$
which has been introduced by Buchweitz \cite{Buc}.
Here, $\mod R$ stands for the category of finitely generated left $R$-modules and $\mathsf{D^b}(\mathsf{mod} R)$ its bounded derived category, and $\kb(\proj R)$ the homotopy category of finitely generated projective $R$-modules.
The singularity categories have been deeply investigated from algebro-geometric and representation-theoretic motivations \cite{Che(1), IW, Kra, Ste, ZZ} and recently connected to the Homological Mirror Symmetry Conjecture by Orlov \cite{Orl04}.

One of the important subjects in representation theory of rings is to classify rings up to certain category equivalences.
For example, left noetherian rings $R$ and $S$ are said to be:
\begin{itemize}
\item {\it Morita equivalent} if $\mod R \cong \mod S$ as abelian categories,
\item {\it derived equivalent} if $\db(\mod R) \cong \db(\mod S)$ as triangulated categories,
\item {\it singularly equivalent} if $\ds(R) \cong \ds(S)$ as triangulated categories.
\end{itemize}
It is well-known that these equivalences have the following relations:
$$
\mbox{Morita equivalence} \Rightarrow \mbox{derived equivalence} \Rightarrow \mbox{singular equivalence}.
$$
Complete characterizations of Morita equivalence and derived equivalence have already been obtained in \cite{Mor, Ric}, while singular equivalence is quite difficult to characterize even in the case of commutative rings.
Indeed, only a few examples of singular equivalences of commutative noetherian rings are known.
For all of such known examples, the singular loci of rings are homeomorphic.
Thus, it is natural to ask the following question.

\begin{ques}\label{ques}
Let $R$ and $S$ be commutative noetherian rings.
Are their singular loci homeomorphic if $R$ and $S$ are singularly equivalent?
\end{ques}

The main purpose of this paper is to show that this question is affirmative for certain classes of commutative noetherian rings.
To be precise, we shall prove the following theorem.

\begin{thm}[Theorem \ref{mainds}] \label{main}
	Let $R$ and $S$ be commutative noetherian local rings that are locally hypersurfaces on the punctured spectra.
	Assume that $R$ and $S$ are either
	\begin{enumerate}[\rm(a)]
		\item
		Gorenstein rings, or
		\item 
		Cohen-Macaulay rings with quasi-decomposable maximal ideal.
	\end{enumerate}
	If $R$ and $S$ are singularly equivalent, then their singular loci $\sing R$ and $\sing S$ are homeomorphic.
\end{thm}
Here, we say that an ideal $I$ of a commutative noetherian ring $R$ is {\it quasi-decomposable} if there is an $R$-regular sequence $\underline{x}$ in $I$ such that $I/(\underline{x})$ is decomposable as an $R$-module.

The key role to prove this theorem is played by the {\it support theory} for triangulated categories.
The support theory has been developed by Balmer \cite{Bal02, Bal05} for tensor triangulated categories and is a powerful tool to show such a reconstruction theorem.
However, singularity categories do not have tensor triangulated structure in general.
For this reason, we develop the support theory {\it without} tensor structure which is motivated by Balmer's work \cite{Bal05}.

We have considered only topological structures of singular loci so far.
Of course, the topological structure gives us much less information than the scheme structure.
For instance, we can not distinguish isolated singularities via topological structures of those singular loci.
Therefore, reconstructing scheme structures of singular loci from singularity categories is a more interesting problem.
For a Henselian Gorenstein local ring with an isolated singularity (i.e., $\dim \sing R = 0$), its singular locus can be reconstructed from its singularity category. 
For details, please see Remark \ref{rem5}.
Thus, the next step we have to consider is the case of commutative noetherian rings with $\dim \sing R = 1$.
In this paper, we consider the following particular singularities.

Let $k$ be an algebraically closed field of characteristic $0$.
We say that a hypersurface singularity $R = k[[x_0, x_1,\ldots, x_d]]/(f)$ is of type $(\a_\infty)$ or $(\d_\infty)$ if $f$ is
\begin{align*}
(\a_\infty)\quad & x_0^2 + x_2^2 + \cdots +x_d^2, \mbox{ or }\\
(\d_\infty)\quad & x_0^2x_1+ x_2^2 + \cdots + x_d^2.
\end{align*}
For these singularities, we obtain the following result.

\begin{thm} [Theorem \ref{schstr}]
Let $R$ be a hypersurface of type $(\a_\infty)$ or $(\d_\infty)$ and $M \in \ds(R)$ an indecomposable object with $\sing R = \ssupp_R(M)$.
Then $\End_{\ds(R)}(M)$ is a commutative ring and there is a scheme isomorphism
$$
\sing R \cong \spec \End_{\ds(R)}(M).
$$
Here, we think of $\sing R$ as the reduced closed subscheme of $\spec R$.
\end{thm}
The symbol $\ssupp_R(M)$ stands for the {\it singular support} of $M$, which is by definition the set of prime ideals $\p$ of $R$ such that $M_\p \not\cong 0$ in $\ds(R_\p)$. 
As it will be proved in Theorem \ref{mainspc}, the condition $\sing R = \ssupp_R(M)$ is characterized by using the triangulated structure on $\ds(R)$.
Thus, this theorem states that the scheme structure on $\sing R$ can be reconstructed from $\ds(R)$.

The organization of this paper is as follows.
In section 2, we introduce the notions of a support data and a classifying support data for a given triangulated category and develop the support theory without tensor structure. 
In section 3, we connect the results obtained in section 2 with the support theory for tensor triangulated categories and study reconstructing the topologies of the Balmer spectra without tensor structure.
Moreover, we give some applications to algebraic geometry and modular representation theory. 
In section 4, we prove Theorem \ref{main} and give examples of commutative rings which are not singularly equivalent. 
In section 5, we consider reconstructing scheme structures of singular loci from singularity categories.
Actually, we prove that the singular locus of a hypersurface local ring, which is countable Cohen-Macaulay representation type, is appears as the Zariski spectrum of the endomorphism ring of some objects of the singularity category.

Throughout this paper, all categories are assumed to be essentially small.
For two triangulated category $\T$, $\T'$ (resp. topological spaces $X$, $X'$), the notation $\T \cong \T'$ (resp. $X \cong X'$) means that $\T$ and $\T'$ are equivalent as triangulated categories (resp. $X$ and $X'$ are homeomorphic) unless otherwise specified.

\section{The support theory without tensor structure}
In this section, we discuss the support theory for triangulated categories without tensor structure.
Throughout this section, $\T$ denotes a triangulated category with shift functor $\Sigma$.

First of all, let us recall some basic definitions which are used in this section.

\begin{dfn}
Let $X$ be a topological space and $\T$ a triangulated category.
\begin{enumerate}[\rm(1)]
\item 
We say that $X$ is {\it sober} if every irreducible closed subset of $X$ is the closure of exactly one point.
\item 
We say that $X$ is {\it noetherian} if every descending chain of closed subspaces stabilizes.
\item
We say that a subset $W$ of $X$ is {\it specialization-closed} if it is closed under specialization, namely if an element $x$ of $X$ belongs to $W$, then the closure $\overline{\{x\}}$ is contained in $W$.
Note that $W$ is specialization-closed if and only if it is a union of closed subspaces of $X$.
\item 
We say that a non-empty additive full subcategory $\X$ of $\T$ is {\it thick} if it satisfies the following conditions:
\begin{enumerate}[\rm(i)]
	\item closed under taking shifts: $\Sigma \X = \X$.
	\item closed under taking extensions: for a triangle $L \to M \to N \to \Sigma L$ in $\T$, if $L$ and $N$ belong to $\X$, then so does $M$.
	\item closed under taking direct summands: for two objects $L, M$ of $\T$, if the direct sum $L \oplus M$ belongs to $\X$, then so do $L$ and $M$. 
\end{enumerate} 
For a subcategory $\X$ of $\T$, denote by $\thick_{\T} \X$ the smallest thick subcategory of $\T$ containing $\X$.

\end{enumerate}
\end{dfn}

We introduce the notion of a support data for a triangulated category.
\begin{dfn} \label{sd}
Let $\T$ be a triangulated category.
A {\it support data} for $\T$ is a pair $(X, \sig)$ where $X$ is a topological space and $\sigma$ is an assignment which assigns to an object $M$ of $\T$ a closed subset $\sig(M)$ of $X$
satisfying the following conditions:
\begin{enumerate}[\rm(1)]
\item
$\sig(0)= \emptyset$. 
\item
$\sig(\Sigma^n M)=\sig(M)$ for any $M \in \T$ and $n \in \Z$. 
\item
$\sig(M \oplus N) = \sig(M) \cup \sig(N)$ for any $M, N \in \T$.
\item
$\sig(M) \subseteq \sig(L) \cup \sig(N)$ for any triangle $L \to M \to N \to \Sigma L$ in $\T$.
\end{enumerate}
\end{dfn}

Support data naturally appear in various areas of algebras.

\begin{ex}\label{exam}
\begin{enumerate}[\rm(1)]
	\item
	Let $R$ be a commutative noetherian ring.
	For $M \in \ds(R)$, we define the {\it singular support} of $M$ by
	$$
	\ssupp_R(M):=\{\p \in \sing R \mid M_\p \not\cong 0 \mbox{ in } \ds(R_\p)\}.
	$$
	Then $(\sing R, \ssupp_R)$ is a support data for $\ds(R)$.
    Indeed, it follows from \cite[Theorem 1.1]{AIL} and \cite[Lemma 4.5]{BM} that $\ssupp_R(M)$ is a closed subset of $\sing R$ and that $\ssupp_R$ satisfies the condition (1) in Definition \ref{sd}.
	The remained conditions $(2)$-$(4)$ are clear because the localization functor $\ds(R) \to \ds(R_\p)$ is exact.
	
	Assume that $R$ is Gorenstein.
	Denote by $\cm(R)$ the category of maximal Cohen-Macaulay $R$-modules (i.e., modules $M$ satisfying $\Ext_R^{i}(M, R) = 0$ for all integers $i >0$).
	Recall that the stable category $\lcm(R)$ of $\cm(R)$ is the category whose objects are the same as $\cm(R)$ and the set of morphisms from $M$ to $N$ is given by
	$$
	\lHom_R(M, N) := \Hom_R(M, N)/\mathsf{P}_R(M, N),
	$$ 
	where $\mathsf{P}_R(M, N)$ consists of all $R$-linear maps from $M$ to $N$ factoring through some free $R$-module.
	Then the stable category $\lcm(R)$ has the structure of a triangulated category; see \cite{Hap}.
	Moreover, the natural inclusion induces a triangle equivalence $ F:\lcm(R) \xrightarrow{\cong} \ds(R)$ by \cite{Buc}.
	Thus we obtain the support data $(\sing R, \lsupp_R)$ for $\lcm(R)$ by using this equivalence.
	Here, 
	$$
	\lsupp_R(M) := \ssupp_R(F(M)) = \{\p \in \sing R \mid M_\p \not\cong 0 \mbox{ in } \lcm(R_p)\}
	$$
	for $M \in \lcm(R)$.
\item
Let $X$ be a noetherian scheme.
For $\F \in \dpf(X)$, we define the {\it cohomological support} of $\F$ by
$$
\supp_X(\F):=\{x \in X \mid \F_x \not\cong 0 \mbox{ in } \dpf(\O_{X, x})\}.
$$
Then, $\supp_X(\F) = \bigcup_{n \in \Z}\supp_X(\H^n(\F))$ is a finite union of supports of coherent $\O_X$-modules and hence is a closed subspace of $X$.
Moreover, $(X, \supp_X)$ is a support data for $\dpf(X)$ because the localization is exact.
For details, please see \cite{Th}.
\item
Let $k$ be a field of characteristic $p > 0$ and $G$ a finite group such that $p$ divides the order of $G$.
Then as in the case of Gorenstein rings, we can define the stable category $\lmod kG$ of $\mod kG$ and it is also a triangulated category.

We denote by
$$
\H^*(G; k) = \begin{cases} 
\oplus_{i \in \Z} \H^i(G; k) & p=2 \\
\oplus_{i \in 2\Z} \H^i(G; k) & p:\mbox{odd}
\end{cases}
$$
the direct sum of cohomologies of $G$ with coefficient $k$.
Then $\H^*(G; k)$  has the structure of a graded-commutative noetherian ring by using the cup product and we can consider its homogeneous prime spectrum $\Proj \H^*(G; k)$.
Denote by $V_G(M)$ the {\it support variety} for a finitely generated $kG$-module $M$ which is a closed space of $\Proj \H^*(G; k)$.
Then the pair $(\Proj \H^*(G; k), V_G)$ becomes a support data for $\lmod kG$.
For details, please refer to \cite[Chapter 5]{Ben}.

\end{enumerate}
\end{ex}

\begin{rem} \label{1rem}
Actually, the above examples of support data satisfy the following stronger condition:

 $(1')$ $\sig(M)=\emptyset$ if and only if $M \cong 0$. 
\end{rem}

Let us fix the following notations:

\begin{nota}
Let $\T$ be a triangulated category and $X$ a topological space.
Then we set:
\begin{itemize}
\item
$\th(\T) :=\{\mbox{thick subcategories of } \T \}$,
\item
$\spcl(X) :=\{\mbox{specialization closed subsets of } X\}$,
\item
$\cl(X) :=\{\mbox{closed subsets of } X\}$,
\item
$\irr(X) :=\{\mbox{irreducible closed subsets of } X\}$.
\end{itemize}
\end{nota}

Let $(X, \sig)$ be a support data for $\T$, $\X$ a thick subcategory of $\T$, and $W$ a specialization-closed subset of $X$. 
Then one can easily check that $f_\sig(\X):= \sig(\X):=\bigcup_{M \in \X}\sig(M)$ is a specialization-closed subset of $X$ and $g_\sig(W):= \sig^{-1}(W):= \{M \in \T \mid \sig(M) \subseteq W \}$ is a thick subcategory of $\T$.
Therefore, we obtain two order-preserving maps
$$
\xymatrix{
	\th(\T) \ar@<0.5ex>[r]^-{f_{\sig}} &
	\spcl(X) \ar@<0.5ex>[l]^-{g_{\sig}} 
}
$$
with respect to the inclusion relations.

\begin{dfn}\label{clssp}
Let $(X, \sig)$ be a support data for $\T$.
Then we say that $(X, \sig)$ is a {\it classifying support data} for $\T$ if 
\begin{enumerate}[\rm(i)]
\item 
$X$ is a noetherian sober space, and 
\item 
the above maps $f_\sig$ and $g_\sig$ are mutually inverse bijections:
$$
\xymatrix{
\th(\T) \ar@<0.5ex>[r]^-{f_{\sig}} &
\spcl(X). \ar@<0.5ex>[l]^-{g_{\sig}} 
}
$$
\end{enumerate}

One can easily check that the classifying support data $(X, \sig)$ for $\T$ satisfies the condition $(1')$ in Remark \ref{1rem}.
\end{dfn}

Here, we have to mention that our definition of a (classifying) support data is motivated by the (classifying) support data for a tensor triangulated category which was introduced and discussed by Balmer \cite{Bal05}.
In this paper, Balmer's (classifying) support data will be called {\em tensorial} (classifying) support data and the relationship between them and our (classifying) support data will be discussed in the next section.

Every classifying support data automatically satisfies the following realization property. 

\begin{lem}\label{tprin}
	Let $(X, \sig)$ be a classifying support data for $\T$.
	Then for any closed subset $Z$ of $X$, there is an object $M$ of $\T$ such that $Z=\sig(M)$. 
\end{lem}

\begin{proof}
	Since $X$ is a noetherian sober space and $\sig(M) \cup \sig(N)=\sig(M \oplus N)$, we may assume that $Z=\overline{\{x\}}$ for some $x \in X$.
	From the assumption, one has $Z=f_\sig g_\sig(Z) = \bigcup_{M \in g_\sig (Z)} \sig(M)$. 
	Hence, there is an element $x$ of $\sig(M)$ for some $M \in g_\sig(Z)$.
	Then we obtain $x \in \sig(M) \subseteq Z=\overline{\{x\}}$ and this implies that $\sig(M)=\overline{\{x\}}=Z$.
\end{proof}

Next, let us introduce the notion of a spectrum of a triangulated category.

\begin{dfn}
\begin{enumerate}[\rm(1)]
\item
We say that a thick subcategory $\X$ of $\T$ is {\it principal} if there is an object $M$ of $\T$ such that $\X=\thick_{\T} M$.
Denote by $\pth(\T)$ the set of all principal thick subcategories of $\T$.
\item
We say that a non-zero principal thick subcategory $\X$ of $\T$ is {\it irreducible} if $\X=\thick_{\T}(\X_1 \cup \X_2)$ ($\X_1, \X_2 \in \pth(\T)$) implies that $\X_1=\X$ or $\X_2=\X$.
Denote by $\spec \T$ the set of all irreducible thick subcategories of $\T$.
\item
For $M \in \T$, set 
$$
\supp_\T(M) := \{\X \in \spec \T \mid \X \subseteq \thick M  \} \subseteq \spec \T.
$$
We consider a topology on $\spec \T$ with closed subbasis $\{\supp_\T(M) \mid M \in \T \}$.
We call this topological space the {\it spectrum} of $\T$.
\end{enumerate}
\end{dfn}

The following lemma shows that by using classifying support data, we can also classify principal thick subcategories and irreducible thick subcategories.

\begin{lem}\label{cls}
Let $(X, \sig)$ be a classifying support data for $\T$, then the one-to-one correspondence
$$
\xymatrix{
\th(\T) \ar@<0.5ex>[r]^-{f_{\sig}} &
\spcl(X), \ar@<0.5ex>[l]^-{g_{\sig}} 
}
$$
which restricts to one-to-one correspondences
$$
\xymatrix{
\pth(\T) \ar@<0.5ex>[r]^-{f_{\sig}} &
\cl(X), \ar@<0.5ex>[l]^-{g_{\sig}} 
}
$$
$$
\xymatrix{
\spec \T \ar@<0.5ex>[r]^-{f_{\sig}} &
\irr(X). \ar@<0.5ex>[l]^-{g_{\sig}} 
}
$$
\end{lem}

\begin{proof}
Note that $f_\sig(\thick_{\T} M)= \sig(M)$ for any $M \in \T$.
Therefore, the injective map $f_\sig: \th(\T) \to \spcl(X)$ induces a well defined injective map $f_\sig: \pth(\T) \to \cl(X)$.
The surjectivity has been already proved in Lemma \ref{tprin}.

Next, we show the second one-to-one correspondence.
For $\X_1, \X_2 \in \th(\T)$, one has 
\begin{align}
f_\sig(\thick_{\T}(\X_1 \cup \X_2)) &= \bigcup_{M \in \thick_{\T}(\X_1 \cup \X_2)}\sig(M) \tag{1} \\
&= \bigcup_{M \in \X_1 \cup \X_2}\sig(M) \notag \\
&=(\bigcup_{M \in \X_1}\sig(M)) \cup (\bigcup_{M \in \X_2}\sig(M)) \notag \\
&= f_\sig(\X_1) \cup f_\sig(\X_2). \notag 
\end{align}
On the other hand, for $Z_1, Z_2 \in \spcl(X)$, one has
\begin{align*}
f_\sig(\thick_{\T}(g_\sig(Z_1) \cup g_\sig(Z_2))) &= f_\sig(g_\sig(Z_1)) \cup f_\sig(\sig(Z_2)) \\ 
&= Z_1 \cup Z_2.
\end{align*}
Applying $g_{\sig}$ to this equality, we get 
\begin{equation}
\thick_{\T}(g_\sig(Z_1) \cup g_\sig(Z_2))=g_{\sig}(Z_1 \cup Z_2). \tag{2}
\end{equation}

Let $W$ be an irreducible closed subset of $X$.
Assume $g_\sig(W) = \thick_\T(\X_1 \cup \X_2)$ for some $\X_1, \X_2 \in \pth(\T)$.
Then, from the above equality (1), we obtain
$$
W = f_\sig(g_\sig(W)) = f_\sig(\thick_\T(\X_1 \cup \X_2)) = f_\sig(\X_1) \cup f_\sig(\X_2).
$$ 
Since $W$ is irreducible, $f_\sig(\X_1) = W$ or $f_\sig(\X_2) = W$ and hence $\X_1 = g_\sig(f_\sig(\X_1))=g_\sig(W)$ or $\X_2 = g_\sig(f_\sig(\X_2))=g_\sig(W)$.
This shows that $g_\sig(W)$ is irreducible.

Conversely, take a irreducible thick subcategory $\X$ of $\T$ and assume $f_\sig(\X) = Z_1 \cup Z_2$ for some closed subsets $Z_1, Z_2$ of $X$.
From the above equality (2), we get
$$
\X = g_\sig(f_\sig(\X)) = g_\sig(Z_1 \cup Z_2) = \thick_\T(g_\sig(Z_1) \cup g_\sig(Z_2)).
$$
Since $\X$ is irreducible, $\X = g_\sig(Z_1)$ or $\X = g_\sig(Z_2)$ and therefore, $Z_1 = f_\sig(g_\sig(Z_1)) = f_\sig(\X)$ or $Z_2 = f_\sig(g_\sig(Z_2)) = f_\sig(\X)$.
Thus, $f_\sig(\X)$ is irreducible. 

These observations show the second one-to-one correspondence.
\end{proof} 

The following theorem is the main result of this section which is an analogous result to \cite[Theorem 5.2]{Bal05}.

\begin{thm}\label{main2}
Let $\T$ be an essentially small triangulated category which admits a classifying support data $(X, \sig)$.
Then there is a homeomorphism $\varphi: X \xrightarrow{\cong} \spec \T$ which restricts to a homeomorphism $\varphi: \sig(M) \xrightarrow{\cong} \supp_\T (M)$ for each $M \in \T$.
\end{thm}

\begin{proof}
First note that for a topological space $X$, the natural map $\iota_X : X \to \irr(X), x \mapsto \overline{\{x\}}$ is bijective if and only if $X$ is sober.

From Lemma \ref{cls}, we have a bijective map
$$
\varphi: X \xrightarrow{\iota_X} \irr(X) \xrightarrow{g_\sig} \spec \T.
$$
As $\X \subseteq \thick M \Leftrightarrow f_\sig(\X) \subseteq f_{\sig}(\thick M) =\sig(M)$, this bijection restricts to
$$
\varphi: \sig(M) \xrightarrow{\cong} \supp_\T M.
$$
Since the topology on $\spec \T$ is given by the closed subbasis $\{\supp_\T M \mid M \in \T \}$, we conclude that $\varphi : X \to \spec \T$ is a homeomorphism.
\end{proof}

Note that $\spec \T$ is determined by the triangulated structure of $\T$.
Therefore, if two triangulated categories $\T$ and $\T'$ are equivalence as triangulated categories, then $\spec \T$ and $\spec \T'$ are homeomorphic. 
From this observation, we obtain the following corollary, which shows that classifying support data is unique up to homeomorphism.

\begin{cor}\label{cor}
Consider the following settings:
\begin{itemize}
\item
$\T$ and $\T'$ are triangulated categories.
\item 
$(X, \sig)$ and $(Y, \tau)$ are classifying support data for $\T$ and $\T'$.
\end{itemize}
Suppose that there is a triangle equivalence $F: \T \xrightarrow{\cong} \T'$.
Then there is a homeomorphism $\varphi:X \xrightarrow{\cong} Y$ which restricts to a homeomorphism
$
\varphi: \sig(M) \xrightarrow{\cong} \tau(F(M))
$ 
for each $M \in \T$.
\end{cor}

\begin{proof}
From the assumption, $F$ induces a one-to-one correspondence 
$$
\tilde{F}: \th(\T) \xrightarrow{\cong} \th(\T'), \,\, \X \mapsto \tilde{F}(\X),
$$
where $\tilde{F}(\X):=\{N \in \T' \mid \exists M \in \X \mbox{ such that } N \cong F(M)\}$.
For an object $M$ of $\T$, set $\tau^F(M):= \tau(F(M))$.
Then we can easily verify that the pair $(Y, \tau^F)$ is a support data for $\T$.
Furthermore, it becomes a classifying support data for $\T$.
Indeed, for $\X \in  \th(\T)$ and $W \in \cl(Y)$, we obtain
$$
f_{\tau^F}(\X) = \bigcup_{M \in \X} \tau^F(M) = \bigcup_{M \in \X} \tau(F(M)) = \bigcup_{N \in \tilde{F}(\X)} \tau(N) = f_\tau(\tilde{F}(\X)),
$$
\begin{align*}
\tilde{F}(g_{\tau^F}(W)) &= \tilde{F}(\{M \in \T \mid \tau^F(M) \subseteq W \}) \\
&=\{N \in \T' \mid \tau(N) \subseteq W\} =g_\tau(W).
\end{align*}
From these equalities, we get equalities $f_{\tau^F}= f_\tau \circ \tilde{F}$ and $\tilde{F} \circ g_{\tau^F} = g_\tau$ and thus $f_{\tau^F}$ and $g_{\tau^F}$ give mutually inverse bijections between $\th(\T)$ and $\cl(Y)$.
Consequently, we obtain two classifying support data $(X, \sig)$ and $(Y, \tau^F)$ for $\T$, and hence both $X$ and $Y$ are homeomorphic to $\spec \T$ by Theorem \ref{main2}.
\end{proof}

\section{Comparison with tensor triangulated structure}
In this section, we discuss relation between the support theory we discussed in section 2 and the support theory for tensor triangulated categories.

Recall that a tensor triangulated category $(\T, \otimes, \one)$ consists of a triangulated category $\T$ together with a symmetric monoidal tensor product $\otimes$ with unit object $\one$ which is compatible with the triangulated structure of $\T$.
For the precise definition, please refer to \cite[Appendix A]{HPS}.

\begin{ex} \label{ttc}
\begin{enumerate}[\rm(1)]
\item 
Let $X$ be a noetherian scheme.
Then $(\dpf(X), \ltensor_{\O_X}, \O_X)$ is a tensor triangulated category.
Here, $\ltensor_{\O_X}$ denotes the derived tensor product.
\item 
Let $k$ be a field and $G$ a finite group.
Then $(\lmod kG, \otimes_k, k)$ is a tensor triangulated category.
\end{enumerate}
\end{ex}

Throughout this section, fix a tensor triangulated category $(\T, \otimes, \one)$.
We begin with recalling some basic definitions which are used in the support theory of tensor triangulated categories.

\begin{dfn}
\begin{enumerate}[\rm(1)]
\item
A full subcategory $\X$ of $\T$ is called a {\it thick tensor ideal} if it is a thick subcategory of $\T$ and is closed under the action of $\T$ by $\otimes$: $M \otimes N \in \X$ for any $M \in \X$ and $N \in \T$.
For a subcategory $\X$ of $\T$, denote by $\langle \X \rangle$ the smallest thick tensor ideal of $\T$ containing $\X$.
\item
For a thick subcategory $\X$ of $\T$, define its {\it radical} by
$$
\sqrt{\X}:= \{M \in \T \mid \exists n > 0 \mbox{ such that } M^{\otimes n} \in \X\}.
$$ 
Here, $M^{\otimes n}$ denotes the $n$-fold tensor product of $M$.
By \cite[Lemma 4.2]{Bal05}, the radical of a thick subcategory is always a thick tensor ideal.

A thick tensor ideal $\X$ of $\T$ is called {\it radical} if it satisfies $\X = \sqrt{\X}$.

\item
A thick tensor ideal $\X$ of $\T$ is called {\it prime} if it satisfies
$$
M \otimes N \in \X \Rightarrow M \in \X \mbox{ or } N \in \X.
$$
Denote by $\spc \T$ the set of all prime thick tensor ideals of $\T$.
\item 
For $M \in \T$, the {\it Balmer support} of $M$ is defined as $\spp_{\T} M := \{\P \in \spc (\T, \otimes) \mid M \notin \P\}$.
The set $\spc (\T, \otimes)$ is a topological space with closed basis $\{\spp_{\T} M \mid M \in \T\}$ and call it the {\it Balmer spectrum} of $\T$.
\item 
Let $X$ be a topological space.
We say that a subset $W$ of $X$ is a {\it Thomason subset} if it is a union of closed subsets whose complements are quasi-compact.
Denote by $\thom(X)$ the set of all Thomason subsets of $X$.
Note that $\thom(X) \subseteq \spcl(X)$.
\end{enumerate}
\end{dfn}

We say that a support data $(X, \sig)$ for $\T$ is {\it tensorial} if it satisfies:
$$
\sig(M \otimes N) = \sig(M) \cap \sig(N)
$$
for any $M, N \in \T$.
In \cite{Bal05}, tensorial support data are called simply support data.
Then $g_\sig(W)$ is a radical thick tensor ideal of $\T$ for every specialization-closed subset $W$ of $X$. 
We say that a tensorial support data $(X, \sig)$ is {\it classifying} if $X$ is a noetherian sober space and there is a one-to-one correspondence:
$$
\xymatrix{
	\{\mbox{radical thick tensor ideals of } \T\} \ar@<0.5ex>[r]^-{f_{\sig}} &
	\spcl(X). \ar@<0.5ex>[l]^-{g_{\sig}} 
}
$$

Balmer showed the following celebrated result:
\begin{thm}\cite[Lemma 2.6, Theorem 4.10]{Bal05}
\begin{enumerate}[\rm(1)]
\item 
The pair $(\spc (\T, \otimes), \spp_{\T})$ is a tensorial support data for $\T$.
\item
There is a one-to-one correspondence:
$$
\xymatrix{
	\{\mbox{radical thick tensor ideals of } \T\} \ar@<0.5ex>[r]^-{f_{\spp}} &
	\thom(\spc (\T, \otimes)). \ar@<0.5ex>[l]^-{g_{\spp}} 
}
$$
\end{enumerate}
\end{thm}

\begin{rem}
If a topological space $X$ is noetherian, then every specialization-closed subset of $X$ is Thomason.
Therefore, the above theorem shows that $(\spc (\T, \otimes), \spp)$ is a classifying tensorial support data for $\T$ provided $\spc (\T, \otimes)$ is noetherian.
\end{rem}

Recall that a tensor triangulated category $\T$ is {\it rigid} if
\begin{enumerate}
\item 
the functor $M \otimes -: \T \to \T$ has a right adjoint $F(M, -): \T \to \T$ for each $M \in \T$ and
\item
every object $M$ is {\it strongly dualizable} (i.e., the natural map $F(M, \one) \otimes N \to F(M, N)$ is an isomorphism for each $N \in \T$).
\end{enumerate}
If $\T$ is rigid, then $(\spc (\T, \otimes), \spp_\T)$ satisfies the stronger condition.

\begin{lem}\label{1rgdsup}
Assume that $\T$ is rigid.
Then the support data $(\spc (\T, \otimes), \spp_\T)$ satisfies the condition $(1')$ in Remark \ref{1rem}.
\end{lem}

\begin{proof}
Take an object $M \in \T$ with $\spp(M) = \emptyset$.
By \cite[Corollary 2.4]{Bal05}, there is a positive integer $n$ such that $M^{\otimes n} \cong 0$ where $M^{\otimes n}$ stands for the $n$-fold tensor product of $n$-copies of $M$.
On the other hand, by \cite[Lemma A 2.6]{HPS}, $M^{\otimes i}$ belongs to $\ttthick_{\T}(M^{\otimes 2i})$ for any positive integer since every object is strongly dualizable.
Therefore, by using induction, we conclude that $M \cong 0$.
\end{proof}

Note that a tensorial classifying support data for $\T$ is a classifying tensorial support data for $\T$.
Indeed, for a tensorial classifying support data $(X, \sig)$ for $\T$ and $\X \in \th(\T)$, we obtain an equalities
$$
\X = g_\sig(f_\sig(\X)) = g_\sig(f_\sig(\sqrt{\ttthick{\X}})) = \sqrt{\ttthick \X}.
$$
The following lemma gives a criterion for the converse implication of this fact.

\begin{lem}\label{tensor}
Let $(X, \sig)$ be a classifying tensorial support data for $\T$.
Suppose that $\T$ is rigid.
Then the following are equivalent:
\begin{enumerate}[\rm(1)]
\item
There is a one-to-one correspondence:
$$
\xymatrix{
\th(\T) \ar@<0.5ex>[r]^-{f_{\sig}} &
\spcl(X). \ar@<0.5ex>[l]^-{g_{\sig}} 
}
$$
\item
$(X, \sig)$ is a classifying support data for $\T$. 
\item 
Every thick subcategory of $\T$ is a thick $\otimes$-ideal.
\item 
$\T = \thick_{\T} \one$.
\end{enumerate}
\end{lem}

\begin{proof}
By Lemma \ref{1rgdsup} and Theorem \cite[Theorem 5.2]{Bal05}, $(X, \sig)$ satisfies the condition $(1')$ in Remark \ref{1rem}.
Therefore, $(1)$ and $(2)$ means the same conditions from Remark \ref{3rem}.

$(1) \Rightarrow (3)$: 
From the assumption, every thick subcategory $\X$ of $\T$ is of the form $\X = g_\sig(W)$ for some specialization-closed subset $W$ of $X$.
On the other hand, $g_\sig(W)$ is a radical thick $\otimes$-ideal as $(X, \sig)$ is a tensorial support data.

$(3) \Rightarrow (4)$: 
By assumption, the thick subcategory $\thick_{\T} \one$ is a thick tensor ideal.
Thus, for any $M \in \T$, $M \cong M \otimes \one$ belongs to $\thick_{\T} \one$.

$(4) \Rightarrow (1)$: 
Note that $\one$ is strongly dualizable and the family of all strongly dualizable objects forms a thick subcategory of $\T$ by \cite[Theorem A.2.5 (a)]{HPS}.
Therefore, every object of $\T = \thick_\T \one$ is strongly dualizable.
Thus, for any object $M \in \T$, $M$ belongs to $\ttthick_{\T}( M \otimes M )$ by \cite[Lemma A.2.6]{HPS}.
Then \cite[Proposition 4.4]{Bal05} shows that every thick tensor ideal of $\T$ is radical.

On the other hand, for any thick subcategory $\X$ of $\Y$, one can easily verify that the subcategory
$
\Y:=\{M \in \T \mid M \otimes \X \subseteq \X \}
$
is a thick $\otimes$-ideal of $\T$ containing $\one$. 
Thus, we obtain $\Y = \thick_{\T}\one = \T$ and hence $\X$ is a thick $\otimes$-ideal.

From these discussion, we conclude that every thick subcategory of $\T$ is a radical thick $\otimes$-ideal and this shows the implication $(4) \Rightarrow (1)$.
\end{proof}

The following corollaries are direct consequences of this lemma, Theorem \ref{main2} and Corollary \ref{cor}.

\begin{cor}
Let $\T$ be a rigid tensor triangulated category.
Assume that the Balmer spectrum $\spc (\T, \otimes)$ of $\T$ is noetherian and $\T= \thick_{\T} \one$.
Then, there is a homeomorphism $\varphi: \spc (\T, \otimes) \xrightarrow{\cong} \spec \T$ which restricts to a homeomorphism $\varphi: \spp_\T M \xrightarrow{\cong} \supp_\T M$ for each $M \in \T$.

\end{cor}

\begin{cor}\label{cor2}
Let $(\T, \otimes, \one)$ and $(\T', \otimes', \one')$ be rigid tensor triangulated categories such that
\begin{itemize}
\item 
$\spc (\T, \otimes)$ and $\spc (\T', \otimes')$ are noetherian, and
\item 
$\T$ and $\T'$ are generated by their units $\one$ and $\one'$, respectively.
\end{itemize}
If $\T$ and $\T'$ are equivalent just as triangulated categories, then $\spc (\T, \otimes)$ and $\spc (\T', \otimes')$ are homeomorphic.
\end{cor}



Next, we consider several applications of results we discussed in this section to tensor triangulated categories appeared in Example \ref{ttc}.

Thomason showed the following classification theorem of thick tensor ideas of $\dpf(X)$.

\begin{thm}\cite[Theorem 3.15]{Th}
Let $X$ be a noetherian scheme.
Then $(X, \supp_X)$ is a classifying tensorial support data for $\dpf(X)$.
\end{thm}

As an application of Theorem \ref{main2}, we can reconstruct underlying topological spaces of a certain class of schemes from their perfect derived categories without tensor structure. 

\begin{thm}\label{mainsch}
Let $X$ be a noetherian quasi-affine scheme.
Then there is a homeomorphism
$$
\varphi: X \xrightarrow{\cong} \spec \dpf(X)
$$
which restricts to a homeomorphism
$
\varphi: \supp_X (\F) \xrightarrow{\cong} \supp_{\dpf(X)} (\F)
$
for each $\F \in \dpf(X)$.
\end{thm}

\begin{proof}
First, let me remark that the functor $\F \ltensor_{\O_X} - : \dpf(X) \to \dpf(X)$ has a right adjoint $\rhom_{\O_X}(\F, -): \dpf(X) \to \dpf(X)$ for each $\F \in \dpf(X)$ and moreover $\dpf(X)$ is rigid.

Note that a scheme $X$ is quasi-affine if and only if its structure sheaf $\O_X$ is ample.	
Thus, every thick subcategory of $\dpf(X)$ is thick tensor ideal by \cite[Proposition 3.11.1]{Th}. 
Applying Theorem \ref{main2}, we obtain the result.
\end{proof}

As a direct consequence of the theorem, we have a necessary condition for derived equivalences for noetherian quasi-affine schemes.

\begin{cor}
Let $X$ and $Y$ be noetherian quasi-affine schemes (i.e., open subschemes of affine schemes). 
If $X$ and $Y$ are derived equivalent (i.e., $\dpf(X) \cong \dpf(Y)$ as triangulated categories), then $X$ and $Y$ are homeomorphic.
In particular, topologically determined properties, such as the dimensions and the numbers of irreducible components  of quasi-affine noetherian schemes are preserved by derived equivalences. 
\end{cor}

\begin{rem}
Let $X$ and $Y$ be noetherian schemes.
\begin{enumerate}[\rm(1)]
\item 
By \cite[Proposition 9.2]{Ric}, if $X$ and $Y$ are affine, then a derived equivalence $\dpf(X) \cong \dpf(Y)$ implies that $X$ and $Y$ are isomorphic as schemes. 
\item
By  \cite[Theorem 9.7]{Bal02}, if $\dpf(X)$ and $\dpf(Y)$ are equivalent as tensor triangulated categories, then $X$ and $Y$ are isomorphic as schemes.
\end{enumerate}
\end{rem}

Next consider stable module categories over group rings of finite groups.
In this case, the following classification theorem is given by Benson-Carlson-Rickard for algebraically closed field $k$ and by Benson-Iyengar-Krause for general $k$.

\begin{thm}\cite{BCR, BIK} 
Let $k$ be a field of characteristic $p > 0$ and $G$ a finite group such that $p$ divides the order of $G$.
Then the support data $(\Proj \H^*(G; k), V_G)$ is a classifying tensorial support data for $\lmod kG$.
\end{thm}

Applying Theorem \ref{main2} to this classifying tensorial support data, we obtain the following result:

\begin{thm}\label{cor3}
Let $k$ be a field of characteristic $p$ and $G$ a $p$-group.
Then there is a homeomorphism
$$
\varphi: \Proj \H^*(G; k) \xrightarrow{\cong} \spec \lmod kG
$$
which restricts to a homeomorphism
$
\varphi: \V_G(M) \xrightarrow{\cong} \supp_{\lmod kG} (M)
$
for each $M \in \lmod kG$.
\end{thm}

\begin{proof}
For each $M \in \lmod kG$, the functor $M \otimes_k -: \lmod kG \to \lmod kG$ has a right adjoint $\Hom_k(M, -): \lmod kG \to \lmod kG$ and in addition $\lmod kG$ is rigid.  
Moreover, for a $p$-group $G$, $kG$ has only one simple module $k$.
Therefore, we have $\lmod kG = \thick_{\lmod kG} k$. 
Applying Theorem \ref{main2}, we are done.
\end{proof}

As in the case of perfect derived categories, we have a necessary condition for stable equivalences for group algebras of $p$-groups.

\begin{cor}\label{cor3}
Let $k$ (resp. $l$) be field of characteristic $p$ (resp. $q$), $G$ (resp. $H$) be a finite $p$-group (resp. $q$-group).
If $kG$ and $lH$ are stably equivalent (i.e., $\lmod kG \cong \lmod lH$ as triangulated categories), then $\Proj \H^*(G; k)$ and $\Proj \H^*(H; l)$ are homeomorphic.
\end{cor}

Recall that the {\it $p$-rank} of a finite group $G$ is by definition the maximal rank of elementary abelian $p$-subgroup:
$$
r_p(G):= \sup \{r \mid (\Z/p)^r \subseteq G \}.
$$
Quillen \cite{Qui} showed that the dimension of the cohomology ring $H^*(G; k)$ is equal to the $p$-rank of $G$.
Thus, the $p$-rank is an invariant of stable equivalences:

\begin{cor}\label{prk}
Let $k, l, G, H$ be as in Theorem \ref{cor3}.
Assume that there is a stable equivalence between $kG$ and $lH$, then $r_p(G)=r_q(H)$.
\end{cor}

Other invariants of stable equivalences are studied by Linckelmann \cite{Lin}.

\begin{rem}
Let $G$ and $H$ be a $p$-group and $k$ a field of characteristic $p$.
\begin{enumerate}[\rm(1)]
\item
By \cite[Corollary 3.6]{Lin}, if there exists a stable equivalence between $kG$ and $kH$, then $|G| = |H|$.
\item
By \cite[Corollary 3.2]{Lin}, if there exists a stable equivalence of Morita type between $kG$ and $kH$, then $G \cong H$.
\end{enumerate}
\end{rem}

We will end this section by proving the following corollary which is a combination of the previous corollary and remark. 

\begin{cor}
Let $k$ be a field of characteristic $p$, G and $H$ $p$-groups. 
Suppose that there is a stable equivalence between $kG$ and $lH$.
Then $G$ is an elementary abelian $p$-group if and only if $H$ is an elementary abelian $p$-group.
This is the case, $G$ and $H$ are isomorphic.
\end{cor}

\begin{proof}
Note that for a $p$-group $G$, it is elementary abelian $p$-group if and only if $|G| = p \cdot r_p(G)$.
By Corollary \ref{prk}, one has $r_p(G) = r_p(H)$ and $|G| = |H|$ by \cite[Corollary 3.6]{Lin}.
Thus the statement holds.
\end{proof}


\section{A necessary condition for singular equivalences}

Recall that commutative noetherian rings $R$ and $S$ are said to be {\it singularly equivalent} if their singularity categories are equivalent as triangulated categories.
The only known examples of singular equivalences are the following:

\begin{ex}
\begin{enumerate}[\rm(1)]
\item
If $R \cong S$, then $\ds(R) \cong \ds(S)$.
\item
If $R$ and $S$ are regular, then $\ds(R) \cong 0 \cong \ds(S)$.
\item (Kn\"{o}rrer's periodicity \cite[Chapter 12]{Yos})
Let $k$ be an algebraically closed field of characteristic $0$.
Set $R:=k[[x_0, x_1, ..., x_{d}]]/(f)$ and $S:=k[[x_0, x_1, ..., x_{d}, u, v]]/(f+uv)$.
Then $\ds(R) \cong \ds(S)$.
\end{enumerate}
\end{ex}

\begin{rem}
All of these singular equivalences, the singular loci $\sing R$ and $\sing S$ are homeomorphic.
In fact, the cases $(1)$ and $(2)$ are clear.
Consider the case of $R:=k[[x_0, x_1, ..., x_{d}]]/(f)$ and $S:=k[[x_0, x_1, ..., x_d, u, v]]/(f+uv)$.
Then
\begin{align*}
\sing S &= \V(\partial f/\partial  x_0, \ldots \partial  f/\partial  x_{d}, u, v)  \\
&\cong \spec(S/ (\partial  f/\partial  x_0, \ldots, \partial  f/\partial  x_{d}, u, v)) \\
&\cong \spec(k[[x_0, x_1, ..., x_{d}, u, v]]/(f+uv, \partial  f/\partial  x_0, \ldots, \partial  f/\partial  x_{d}, u, v)) \\
&\cong \spec(k[[x_0, x_1, ..., x_{d}]]/(f, \partial  f/\partial  x_0, \ldots, \partial  f \partial  x_{d}) \\
&\cong \V(\partial  f/\partial x_0, \ldots \partial  f/\partial  x_{d}) = \sing R.
\end{align*}
Here, the first and the last equalities are known as the Jacobian criterion.
\end{rem}

Let me give some definitions appearing in the statement of the main theorem of this section.

\begin{dfn}
Let $(R, \m, k)$ be a commutative noetherian local ring.
\begin{enumerate}
\item
We say that an ideal $I$ of $R$ is {\it quasi-decomposable} if there is an $R$-regular sequence $\underline{x}$ of $I$ such that $I/(\underline{x})$ is decomposable as an $R$-module.
\item
A local ring $R$ is said to be {\it complete intersection} if there is a regular local ring $S$ and an $S$-regular sequence $\underline{x}$ such that the completion $\hat{R}$ of $R$ is isomorphic to $S/(\underline{x})$.
We say that $R$ is a {\it hypersurface} if we can take $\underline{x}$ to be an $S$-regular sequence of length $1$.
\item
A local ring $R$ is said to be {\it locally a hypersurface on the punctured spectrum} if $R_\p$ is a hypersurface for every non-maximal prime ideal $\p$.
\end{enumerate}
\end{dfn}

The following theorem is the main result of this section.

\begin{thm}\label{mainds}
	Let $R$ and $S$ be commutative noetherian local rings that are locally hypersurfaces on the punctured spectra.
	Assume that $R$ and $S$ are either
	\begin{enumerate}[\rm(a)]
		\item
		Gorenstein rings, or
		\item 
		Cohen-Macaulay rings with quasi-decomposable maximal ideal.
	\end{enumerate}
	If $R$ and $S$ are singularly equivalent, then $\sing R$ and $\sing S$ are homeomorphic.
\end{thm}

For a ring $R$ satisfying the condition (b) in Theorem \ref{mainds}, Nasseh-Takahashi \cite[Theorem B]{NT} shows that $(\sing R, \ssupp_R)$ is a classifying support data for $\ds(R)$.
Therefore, the statement of Theorem \ref{mainds} follows from Corollary \ref{cor}.
Therefore, the problem is the case of (a).

For a ring $R$ satisfying the condition (a) in Theorem \ref{mainds}, Takahashi \cite{Tak10} classified thick subcategories of $\ds(R)$ containing the residue field $k$ of $R$ by using the singular locus $\sing R$ and the singular support $\ssupp_R$.
We would like to apply Corollary \ref{cor} also for this case.
The problem is that whether the condition ``containing the residue field $k$" is preserved by singular equivalences.
As we will show later, this condition is actually preserved by singular equivalences for Gorenstein local rings.

For a while, we consider Gorenstein local rings and the stable category $\lcm(R)$ instead of $\ds(R)$.
The following lemma gives a categorical characterization of thick subcategory generated by the residue field.

\begin{lem}
Let $(R, \m, k)$ be a $d$-dimensional Gorenstein local ring.
Then the following are equivalent for $M \in \lcm(R)$.
\begin{enumerate}[\rm(1)]
\item
$\lEnd_{R}(M)$ is an artinian ring.
\item 
$M$ is locally free on the punctured spectrum (i.e., $M_\p \cong 0$ for $\forall \p \neq \m$).
\item
$M \in \thick_{\lcm(R)} \syz^d k $.
\end{enumerate} 
\end{lem}

\begin{proof}
Implications $(3) \Rightarrow (2) \Rightarrow (1)$ is clear and the implication $(2) \Rightarrow (3)$ is shown in \cite[Corollary 2.6]{Tak10}.

The only thing remained to prove is the implication $(1) \Rightarrow (2)$. 
Assume that $\lEnd_R(M)$ is an artinian ring.
Denote by $\mathfrak{a}$ the kernel of the natural map $R \to \lEnd_R(M)$.
Then by \cite[Theorem 1]{Eis}, $R/ \mathfrak{a}$ is an artinian ring and hence $\mathfrak{a}$ is $\m$-primary. 
Therefore, for any non-maximal prime ideal $\p$, the natural map $R_\p \to \lEnd_{R_\p}(M_\p)$ is zero and hence $M_\p \cong 0$.
\end{proof}

Denote by $\lcm_0(R)$ the thick subcategory of $\lcm(R)$ consisting of maximal Cohen-Macaulay modules which are locally free on the punctured spectrum.
As we have shown in the previous lemma, $\lcm_0(R) = \thick_{\lcm(R)} k$ and whose objects are categorically determined objects.
Therefore, we have the following result. 

\begin{prop}\label{lf}
Let $R$ and $S$ be Gorenstein local rings.
If there is a triangle equivalence $\lcm(R) \cong \lcm(S)$, then it induces triangle equivalences
\begin{align*}
\lcm_0(R) &\cong \lcm_0(S), and\\
\lcm(R)/\lcm_0(R) &\cong \lcm(S)/\lcm_0(S).
\end{align*}
\end{prop}

Gathering \cite[Theorem 6.7]{Tak10}, \cite[Theorem B]{NT} and Theorem \ref{main2}, we obtain the following proposition.

\begin{thm}\label{mainspc}
Let $R$ be a noetherian local ring.
\begin{enumerate}[\rm(1)]
\item
If $R$ satisfies the condition $(a)$ in Theorem \ref{mainds}, then there is a homeomorphism
$$\varphi:\sing R \setminus \{\m\} \xrightarrow{\cong} \spec \lcm(R)/\lcm_0(R).$$
\item 
If $R$ satisfies the condition $(b)$ in Theorem \ref{mainds} or is hypersurface, then there is a homeomorphism $$\varphi:\sing R \xrightarrow{\cong} \spec \ds(R)$$ which restricts to a homeomorphism 
$
\varphi:\ssupp_R(M) \xrightarrow{\cong} \supp_{\ds(R)}(M)
$
for each $M \in \ds(R)$.
\end{enumerate}
\end{thm}

\begin{proof}
The statement $(2)$ directly follows from \cite[Theorem 6.7]{Tak10}, \cite[Theorem B]{NT} and Theorem \ref{main2}.

By \cite[Theorem 6.7]{Tak10}, the assignment
$$
\lcm(R)/ \lcm_0(R) \ni M \mapsto \ssupp_R(M) \setminus \{ \m\} \subseteq \sing R \setminus \{\m\}
$$
defines a classifying support data for $\lcm(R)/ \lcm_0(R)$.
Thus the statement $(1)$ follows from Theorem \ref{main2} again.
\end{proof}

Now, the combination of Proposition \ref{lf} and Theorem \ref{mainspc} completes the proof of Theorem \ref{mainds}.


\begin{rem}
For a hypersurface ring $R$, the triangulated category $\ds(R)$ becomes a pseudo tensor triangulated category (i.e., tensor triangulated category without unit).
It is shown by Yu implicitly in the paper \cite{Yu} that for two hypersurfaces $R$ and $S$, if a singular equivalence between $R$ and $S$ preserves tensor products, then $\sing R$ and $\sing S$ are homeomorphic.
Indeed, $\sing R$ is reconstructed from $\ds(R)$ by using its pseudo tensor triangulated structure.
\end{rem}

Since Theorem \ref{mainds} gives a necessary condition for singular equivalences, we can generate many pairs of rings which are not singularly equivalent.
Let us start with the following lemma.

\begin{lem}
Let $R$ be a Gorenstein local ring with only an isolated singularity and $r >1$ an integer.
Then the ring $R[[u]]/(u^r)$ is a Gorenstein local ring which is locally a hypersurface on the punctured spectrum, and $\sing(R[[u]]/(u^r))$ is homeomorphic to $\spec R$.
\end{lem}

\begin{proof}
Set $T:=R[[u]]/(u^r)$.
The natural inclusion $R \to T$ induces a homeomorphism $f: \spec T \xrightarrow{\cong} \spec R$.
Then one can easily check that $P = (f(P), u)T$ for any $P \in \spec T$ and $T_P \cong R_{f(P)}[[u]]/(u^r)$.
Therefore, $T$ is locally a hypersurface on the punctured spectrum and $\sing T = \spec T$.
\end{proof}

\begin{cor}
Let $R$ and $S$ be Gorenstein local rings which have only isolated singularities.
Assume that $\spec R$ and $\spec S$ are not homeomorphic.
Then for any integers $r, s > 1$, one has
$$
\ds(R[[u]]/(u^r)) \not\cong \ds(S[[v]]/(v^s)).
$$
In particular, $\ds(R * R) \not\cong \ds(S * S)$.
Here $R*R$ denotes the trivial extension ring of a commutative ring $R$.
\end{cor}

\begin{proof}
From the above lemma, we obtain
\begin{enumerate}[\rm(1)]
\item
$R[[u]]/(u^r)$ and $S[[v]]/(v^s)$ satisfies the condition (a) in Theorem \ref{mainds}, 
\item
$\sing R[[u]]/(u^r) \cong \spec R$ and $\sing S[[u]]/(v^r) \cong \spec S$ are not homeomorphic.
\end{enumerate}
Thus, we conclude $\ds(R[[u]]/(u^r)) \not\cong \ds(S[[v]]/(v^s))$ by Theorem \ref{mainds}.

The second statement follows from an isomorphism $R * R \cong R[[u]]/(u^2)$.
\end{proof}

The following corollary says that a Kn\"{o}rrer-type equivalence fails over a non-regular ring.

\begin{cor}
Let $S$ be a Gorenstein local ring and $f$ an $S$-regular element.
Assume that $S/(f)$ has an isolated singularity.
Then one has
$$
\ds(S[[u]]/(f, u^2)) \not\cong \ds(S[[u, v, w]]/(f+vw, u^2)).
$$   
\end{cor}

\begin{proof}
$\sing S[[u]]/(f, u^2) \cong \spec S/(f)$ and $\sing S[[u, v, w]]/(f+vw, u^2) \cong \spec S[[v, w]]/(f+vw)$ have different dimensions and hence are not homeomorphic. 
\end{proof}

For the last of this paper, we will show that singular equivalence localizes.

\begin{lem}\label{localization}
Let $R$ be a d-dimensional Gorenstein local ring and $\p$ a prime ideal of $R$.
Then a full subcategory $\X_\p:=\{M \in \ds(R) \mid M_\p \cong 0 \mbox{ in } \ds(R_\p) \}$ is thick and there is a triangle equivalence
$$
\ds(R)/\X_\p \cong \ds(R_\p).
$$
\end{lem}

\begin{proof}
By using the triangle equivalence $\ds(R) \cong \lcm(R)$, we may show the triangle equivalence 
$$
\lcm(R)/\X_\p \cong \lcm(R_\p),
$$
where $\X_\p:=\{M \in \lcm(R) \mid M_\p \cong 0 \mbox{ in } \lcm(R_\p) \}$.	
	
Note that the localization functor $L_\p: \lcm(R) \to \lcm(R_\p), M \mapsto M_\p$ is triangulated.
Since $\X_\p = \Ker L_\p$, $\X_\p$ is a thick subcategory of $\lcm(R)$ and $L_\p$ induces a triangulated functor $\overline{L}_\p : \lcm(R)/\X_\p \to \lcm(R_\p)$.
Thus, we have only to verify that $\overline{L}_\p$ is dense and fully faithful.

(i): $\overline{L_\p}$ is dense.

Let $U$ be an $R_\p$-module.
Take a finite free presentation $R_\p^n \xrightarrow{\del} R_\p^m \to U \to 0$ of $U$.
Then $\del$ can be viewed as an $m \times n$-matrix $(\alpha_{ij})$ with entries in $R_\p$.
Write $\alpha_{ij}= a_{ij}/s$ for some $a_{ij} \in R$ and $s \in R \setminus \p$.
Then the cokernel $M:=\cok((a_{ij}): R^n \to R^m)$ is a finitely generated $R$-module and $M_\p \cong U$.
Since $M_\p$ is a maximal Cohen-Macaulay $R_\p$-module, we obtain isomorphisms
$$
(\syz_R^{-d}\syz_R^d M)_\p \cong \syz_{R_\p}^{-d}\syz_{R_\p}^d M_\p \cong M_\p \cong U
$$ 
in $\lcm(R_\p)$.
This shows that the functor $\overline{L_\p}$ is dense.

(ii): $\overline{L_\p}$ is faithful.
 
Let $\alpha: M \to N$ be a morphism in $\lcm(R)/\X_\p$.
Then $\alpha$ is given by a fraction $f/s$ of morphisms $f:M \to Z$ and $s: N \to Z$ in $\lcm(R)$ such that the mapping cone $C(s)$ of $s$ belongs to $\X_\p$.
Assume $\overline{L}_\p(\alpha) = L_\p(s)^{-1} L_\p(f)= (s_\p)^{-1}f_\p=0$. 
Then $f_\p =0$ in $\lHom_{R_\p}(M_\p, Z_\p)$.
From the isomorphism $\lHom_R(M, Z)_\p \cong \lHom_{R_\p}(M_\p, Z_\p)$, there is $a \in R \setminus \p$ such that $a f=0$ in $\lHom_R(M, Z)$.
Since $a : Z_\p \to Z_\p$ is isomorphism, the mapping cone of the morphism $a : Z \to Z$ in $\lcm(R)$ belongs to $\X_\p$.
Thus, $\alpha = f/s = (af)/(as)=0$ in $\lcm(R)/\X_\p$.
This shows that $\overline{L}_\p$ is faithful.

(iii): $\overline{L_\p}$ is full.

Let $g :M_\p \to N_\p$ be a morphism in $\lcm(R_\p)$ where $M, N \in \lcm(R)$.
By the isomorphism $\lHom_R(M, N)_\p \cong \lHom_{R_\p}(M_\p, N_\p)$, there is a morphism $f: M \to N$ in $\lcm(R)$ and $a \in R \setminus \p$ such that $g = f_\p/a$.
Since the mapping cone of $a : N \to N$ is in $\X_\p$, we obtain a morphism $f/a :M \to N$ in $\lcm(R)/\X_\p$ and $\overline{L}_\p(f/a)=f_\p/a=g$. 
This shows that $\overline{L}_\p$ is full.
\end{proof}

\begin{cor}
Let $R$ and $S$ be Gorenstein local rings which are locally hypersurfaces on the punctured spectra.
If $R$ and $S$ are singularly equivalent, then there is a homeomorphism $\varphi:\sing R \xrightarrow{\cong} \sing S$ such that $R_\p$ and $S_{\varphi(\p)}$ are singularly equivalent for any $\p \in \sing R$.
\end{cor}

\begin{proof}
As in Lemma \ref{localization}, we may consider the category $\lcm(R)$.

Let $F: \lcm(R) \to \lcm(S)$ be a triangle equivalence.
Take a homeomorphism $\varphi: \sing R \to \sing S$ given in Proposition \ref{main2} and Theorem \ref{cor}.
Then by construction, it satisfies
$$
\overline{\{\varphi(\p)\}} = \bigcup_{M \in \lcm(R),\,\, \lsupp_R(M) \subseteq \V(\p)} \lsupp_S F(M)
$$
for each $\p \in \sing R$.
Moreover, the following diagram is commutative:
$$
\begin{CD}
\th_{\TT(\lcm(R))}(\lcm(R)) @>\tilde{F}>> \th_{\TT(\lcm(S))}(\lcm(S)) \\
@Vf_{\lsupp_R}VV @VVf_{\lsupp_S}V \\
\nesc(\sing R) @>>\tilde{\varphi}> \nesc(\sing S),
\end{CD}
$$
where the map $\tilde{F}$ and $\tilde{\varphi}$ are defined by $\tilde{F}(\X):= \{N \in \T' \mid \exists M \in \X \mbox{ such that } N \cong F(M)\}$ and $\tilde{\varphi}(W):=\varphi(W)$, respectively.

Let $\p$ be an element of $\sing R$.
Set $W_\p :=\{\q \in \sing R \mid \q \not\subseteq \p \}$ which is a specialization-closed subset of $\sing R$.
We establish two claims.
\begin{claim}
$g_{\lsupp_R}(W_\p) =\X_\p$.
\end{claim}

\begin{proof}[Proof of Claim 1]
Let $M \in \X_\p$.
Since $M_\p =0$ in $\lcm(R_\p)$, one has $\p \not\in \lsupp_R(M)$.
Thus, $\lsupp_R(M) \subseteq W_\p$ and hence $M \in g_{\lsupp_R}(W_\p)$.

Next, take $M \in g_{\lsupp_R}(W_\p)$.
Then $\lsupp_R(M) \subseteq W_\p$ means that $\p$ does not belong to $\lsupp_R(M)$.
Therefore, $M_\p= 0$ in $\lcm(R_\p)$ and hence $M \in \X_\p$. 
\end{proof}

\begin{claim}
$\varphi(W_\p)= W_{\varphi(\p)}:=\{\q \in \sing S \mid \q \not\subseteq \varphi(\p) \}$.
\end{claim}

\begin{proof}[Proof of Claim 2]
One can easily check that $\varphi$ is order isomorphism with respect to the inclusion relations.
Since $\sing R \setminus W_\p$ has a unique maximal element $\p$, $\varphi(\sing R \setminus W_\p) = \sing S \setminus \varphi(W_\p)$ also has a unique maximal element $\varphi(\p)$.
This shows $\varphi(W_\p)= W_{\varphi(\p)}$.
\end{proof}

From the above two claims, we obtain 
$$
\tilde{F}(\X_{\p})=\tilde{F}(g_{\lsupp_R}(W_{\p})) =g_{\lsupp_S}(\tilde{\varphi}(W_{\p})) = g_{\lsupp_S}(W_{\varphi(\p)})  =  \X_{\varphi(\p)},
$$
where the second equality comes from the above commutative diagram and the last equality is shown by the same proof as Claim 1.
Consequently, the triangle equivalence $F$ induces triangle equivalences:
$$
\lcm(R_\p) \cong \lcm(R) / \X_\p \cong \lcm(S)/ \X_{\varphi(\p)} \cong \lcm(S_{\varphi(\p)}).
$$
\end{proof}

\section{Reconstruction of scheme scheme structure} 
For the last of this paper, we consider reconstructing scheme structure of singular loci from stable category of maximal Cohen-Macaulay modules over hypersurface.

In this section, we always consider a hypersurface $R := S/(f)$, where $S$ is a regular local ring with algebraically closed residue field $k$ and $f$ be a non-zero element of $S$. 
We always consider reduced closed subscheme structure of $\spec R$ on $\sing R$.
The following remarks are an easy observation and a known result.

\begin{rem}[The case of isolated singularities : $\dim \sing R = 0$] \label{rem5}
\begin{enumerate}[\rm(1)]
\item 
Assume that $S$ is Henselian.
Then $\lEnd_{R}(M)$ is a local finite dimensional algebra over the residue field $k$ of $R$ for any non-zero indecomposable object $M$ of $\lcm(R)$.
Thus, $\lEnd_{R}(M)/ J(\lEnd_{R}(M))$ is a finite dimensional division algebra over $k$, where $J(\lEnd_{R}(M))$ denotes the Jacobson radical of $\lEnd_{R}(M)$.
As $k$ is algebraically closed, it is isomorphic to $k$.
Therefore,
$$
\sing R \cong \spec \lEnd_{R}(M)/ J(\lEnd_{R}(M)).
$$
\item (\cite[Corollary 6.5]{Dyc}) 
Assume that $S$ is essentially of finite type.
Note that for a hypersurface singularity $S/(f)$, there is a $\Z/2\Z$-graded DG category $\mathsf{MF}(S, f)$ such that $\mathsf{H}^0(\mathsf{MF}(S, f)) \cong \lcm(R)$.
Then there is an isomorphism of schemes:
$$
\sing R \cong \spec \mathsf{HH}^*(\mathsf{MF}(S, f)).
$$
Here, $\mathsf{HH}^*(\mathsf{MF}(S, f))$ is the Hochschild cohomology ring of $\mathsf{MF}(S, f)$.
\end{enumerate}
\end{rem}

Thus, the next step we have to consider is the case of hypersurfaces with $\dim \sing R = 1$.
In this paper, we consider the following particular singularities.

Let $k$ be an algebraically closed field of characteristic $0$.
We say that a hypersurface singularity $R = k[[x_0, x_1,\ldots, x_d]]/(f)$ is of type $(\a_\infty)$ or $(\d_\infty)$ if $f$ is
\begin{align*}
(\a_\infty)\quad & x_0^2 + x_2^2 + \cdots +x_d^2, \mbox{ or }\\
(\d_\infty)\quad & x_0^2x_1+ x_2^2 + \cdots + x_d^2.
\end{align*}

\begin{rem}\label{rem5-2}
\begin{enumerate}[\rm(1)]
\item 
If $R$ is of type $(\a_\infty)$ or $(\d_\infty)$, then $\dim \sing R = 1$.
\item 
For a hypersurface $R$, $\lcm(R)$ is of countable representation type (i.e., there are only countably many indecomposable objects) if and only if $R$ is isomorphic to the singularity of type $(\a_\infty)$ or $(\d_\infty)$.
Moreover, this is the case, all indecomposable objects of $\lcm(R)$ have been classified completely; see \cite{Sch}. 
\end{enumerate}
\end{rem}

We say that an object $M \in \lcm(R)$ has {\it full-support} if $\supp_{\lcm(R)} (M) = \spec \lcm(R)$.
By Theorem \ref{mainspc}, $M$ is full-support if and only if 
$$
\lsupp_R(M) = \sing R.
$$
Therefore, the condition $\lsupp_R(M) = \sing R$ is categorically determined. 

For a hypersurface singularity of type $(\a_\infty)$ or $(\d_\infty)$, we can reconstruct the scheme structure of the singular locus by using a maximal Cohen-Macaulay module with full support.

\begin{thm}\label{schstr}
Let $R$ be a hypersurface of type $(\a_\infty)$ or $(\d_\infty)$ and $M$ an indecomposable object of $\lcm(R)$ with full-support.
Then $\lEnd_R(M)$ is a commutative ring and one has a scheme isomorphism
$$
\sing R \cong \spec \lEnd_{R}(M).
$$
\end{thm}

\begin{proof}
Thanks to Kn\"{o}rrer's periodicity, we have only to prove in dimension one and two.

\noindent
Case : $\dim R =1$ \& $(\a_\infty)$

In this case, $R= k[[x, y]]/(x^2)$ and which has only one indecomposable maximal Cohen-Macaulay $R$-module $M= R/(x)$ with full-support.
Note that the natural map
$$
R/(x) \to \End_R(M)
$$ 
is an isomorphism.
Since $\Hom_R(R/(x), R) \cong (0 :_R x) = (x)$, one has $\mathsf{P}_R(M, M) = 0$.
Thus, there is a scheme isomorphism
$
\sing R \cong \spec \lEnd_R(M).
$

The proofs for the cases $\dim R = 1$ \& $(\d_\infty)$  and $\dim R = 2$ \& $(\a_\infty)$ are similar since indecomposable maximal Cohen-Macaulay modules with full-supports are cyclic.
Therefore, we have only to prove the following case.

\noindent
Case : $\dim R =2$ \& $(\d_\infty)$

In this case, $R = k[[x, y, z]]/(x^2y + z^2)$ and $M= (x, z) \subseteq R$ is the only one indecomposable maximal Cohen-Macaulay $R$-module with full-support.
Since $z^2 = -x^2y$ in $R$, $R$ is isomorphic to $T \oplus Tz$ as an $T:= k[[x, y]]$-module.

As $M$ is a rank one maximal Cohen-Macaulay $R$-module, we identify $\Hom_R(M, R)$ and $\End_R(M)$ with $R$-submodules of the total quotient field $Q(R)$ of $R$ by 
\begin{align*}
\End_R(M) \subseteq \Hom_R(M, R) &\hookrightarrow \Hom_R(M, R) \otimes Q(R)\\ 
&\cong \Hom_{Q(R)}(M\otimes Q(R), R \otimes Q(R)) \cong Q(R),
\end{align*}
which sends $f$ to $f(x)/x$.
Under this identification, $\Hom_R(M, R)$, $\End_R(M)$ and $\mathsf{P}_R(M, M)$ are given as the followings.

\begin{claim*}
$\End_R(M) = \Hom_R(M, R) = \{a/x \mid az \in xR\} = T + (z/x)T$. 
\end{claim*}

\begin{proof}[Proof of Claim]
By the above identification, one has $\Hom_R(M, R) \subseteq \{a/x \mid a \in R\}$.
On the other hand, for $a/x \in Q(R)$, 
$$
a/x \cdot M \subseteq R \Leftrightarrow a/x \cdot z \in R \Leftrightarrow az \in xR.
$$
Thus the second equality holds.

Take an element $a/x \in Q(R)$ with $az \in xR$ and denote 
$$
a = f + zg, \,\, az = x(u + zv)
$$
with $f, g, u, v \in T$.
Then $zf - x^2yg =  az = xu + xzv$ and hence $-x^2y g = xu$, $f = xv$.
Therefore, $a/x =(xv + zg)/x = v + (z/x)g$.
This shows the inclusion $ \{a/x \mid a \in R\} \subseteq T + (z/x)T$.
Clearly, $ T + (z/x)T$ is contained in $\End_R(M)$.
\renewcommand{\qedsymbol}{$\square$}
\end{proof}

\begin{claim*}
$\mathsf{P}_R(M, M)  = M$.
\end{claim*}

\begin{proof}[Proof of Claim]
Take an element $\xi \in \mathsf{P}_R(M, M)$.
Then there exist $\alpha_i \in \Hom_R(M, R)$, $\beta_i \in M$ $(i = 1, 2, \ldots, r)$ such that
$$
\xi = \beta_1 \alpha_1 +  \beta_2 \alpha_2 + \cdots +  \beta_r \alpha_r.  
$$ 
Denote $\alpha_i = f_i + (z/x)g_i$ and $\beta_i = xu_i + zv_i$ with $f_i, g_i, u_i, v_i \in T$ $(i = 1, 2, \ldots, r)$.
Then, we have
$$
\beta_i \alpha_i = x u_i f_i + z u_i g_i + z v_i f_i + (z^2/x) v_i g_i = x(u_i f_i - y v_i g_i) + z(u_i g_i + v_i f_i).
$$
This shows the inclusion $\subseteq$.

Conversely, for any element $xu+zv \in M$ with $u, v \in T$, it is a composition
$$
\xymatrix{
M \ar[rd]_{u+(z/x)v} \ar[rr]^{xu+zv} && M \\
& R \ar[ru]_{x} &. \\
}
$$
This concludes $\mathsf{P}_R(M, M) = M$.
\renewcommand{\qedsymbol}{$\square$}
\end{proof}

From the above two claims, we obtain
$$
\lEnd_R(M) \cong (T + (z/x)T) / (xT + z T) \cong k[[y]] + (z/x)k[[y]] \cong k[[y, w]]/(w^2+y) \cong k[[y]].
$$
Thus, we get an isomorphism
$
\sing R \cong \spec \lEnd_R(M)
$
of schemes.
\end{proof}

\begin{rem}
As I mentioned in Remark \ref{rem5-2}(2), singular loci of hypersurfaces of type $(\a_\infty)$ and $(\d_\infty)$ are characterized in terms of stable categories of maximal Cohen-Macaulay modules.
The advantage of this theorem is that it gives more explicit description of their singular loci.  
\end{rem}

\begin{ac} 
The author is grateful to his supervisor Ryo Takahashi for many supports and his helpful comments.
\end{ac}



\end{document}